\documentclass [11pt,] {article}
\usepackage{amsmath,amssymb}
\usepackage{amsthm,amsfonts,amscd,epsfig,lineno}
\usepackage{xcolor}
\usepackage{multirow}
\usepackage{cite,epic}
\usepackage{adjustbox}
\usepackage{graphicx}
\usepackage[left=2cm]{geometry}
\usepackage{lscape}

\usepackage{amsmath, amssymb, amsthm}

\newtheorem{theorem}{Theorem}[section]
\newtheorem{lemma}[theorem]{Lemma}

\newtheorem{corollary}[theorem]{Corollary}
\newtheorem{definition}[theorem]{Definition}
\newtheorem{example}[theorem]{Example}
\newtheorem{construction}[theorem]{Construction}
\newcount\refno
\usepackage{array}

\newcommand{\ket}[1]{|#1\rangle}
\newcommand{\qcell}[1]{\(\begin{gathered}#1\end{gathered}\)}

\usepackage{cite}
\numberwithin{equation}{section}

\title { Resolvable quantum Latin squares with maximal cardinality
}

\author {   \footnotesize    Yuyuan Zhang, Haitao Cao\footnote{Corresponding author. \; E-mail address: caohaitao@njnu.edu.cn.  }\\
       \scriptsize    School of Mathematical Sciences, Ministry of Education Key Laboratory for NSLSCS,\\ \scriptsize Nanjing Normal University, Nanjing 210023, China. }
\date{}

\begin {document}
\parindent=0.5cm
\baselineskip=0.6cm
\maketitle
 \begin {abstract}
\baselineskip=0.6cm
In this paper we investigate the existence of a maximal-cardinality \(t\)-resolvable quantum Latin square (\(t\text{-}\text{RQLS}(v)\)). We determine the existence of a maximal-cardinality \(1\text{-}\text{RQLS}(v)\) with \(17\) possible exceptions based on the established research result and construct a maximal-cardinality \(2\text{-}\text{RQLS}(n^2)\) for every integer \(n\geq6\).

\end{abstract}

\section{Introduction}
\label{intro}
A {\it quantum  Latin square}  of order $v$  is a  $v\times v$  square, denoted as $\text{QLS}(v)$, whose entries are unit column vectors from $v$-dimensional  Hilbert space $\mathcal{H}_{v}$, and such that each row and column forms an orthonormal basis of $\mathcal{H}_{v}$. In 2016, Musto and Vicary \cite{Mu2} introduced quantum Latin squares as a quantum-theoretic generalization of classical Latin squares, showing their utility in constructing unitary error bases (UEBs). Subsequent work by Musto \cite{Mu1} established the notions of weakly orthogonal and orthogonal quantum Latin squares, which proved crucial for constructing mutually unbiased bases (MUBs). Goyeneche et al. \cite{Goyeneche} extended quantum Latin structures by introducing quantum Latin cubes, establishing their connections with absolutely maximally entangled (AME) states and $k$-uniform states. In 2021, Nechita and Pillet \cite{Nechita} introduced the concept of quantum Sudoku, a special class of quantum Latin squares.

A \(\text{QLS}(v)\) can be obtained from a classical Latin square by replacing each entry \(i\in [v]\) with the computational basis vector \(|i\rangle\in \mathcal{H}_{v}\), where \([v]=\{0,1,\ldots,v-1\}\). A $\text{QLS}(v)$ is called  \emph{classical}  if all entries  are constrained to the computational basis \( \{|0\rangle, |1\rangle, \dots, |v-1\rangle\} \). In quantum theory, two unit vectors \( |\phi\rangle, |\psi\rangle \in \mathcal{H}_v \)  are regarded as {\it identical} ($|\phi\rangle = |\psi\rangle $) if there exists a real number \( \theta \) such that $ |\phi\rangle = e^{\mathrm{i}\theta} |\psi\rangle$; otherwise, they are considered {\it distinct} ($|\phi\rangle \neq |\psi\rangle $). The {\it cardinality} $c$ of a $\text{QLS}(v)$ is the number of distinct vectors in the array. Clearly the cardinality $c$ of a $\text{QLS}(v)$ satisfies that \( v \leq  c \leq v^2 \).   A $\text{QLS}(v)$ is called \emph{non-classical} if  $c>v$. In particular, a $\text{QLS}(v)$ with maximal cardinality implies $c=v^2$.

The study of cardinalities of quantum Latin squares has attracted significant attention.
In 2021, Paczos et al. \cite{Paczos}  proved the existence of quantum Sudoku of order $v^2$ with  maximum cardinality $c=v^4$  and deriving new mutually unbiased bases. In 2026, Zhang and Cao et al. \cite{Zhang1,Zhang3} almost solved the existence of a \(\text{QLS}(v)\) with $c=v^2$ with $11$ possible exceptions. Then Zang et al. \cite{Zang6} gave a complete solution to this problem for $v\geq 4$.  For the existence of a \(\text{QLS}(v)\) with \(v\leq c\leq v^2\), Zhang and Ji \cite{Zhang2,Zhang5} proved that the cardinality \(c=v+1\) is impossible for any \(\text{QLS}(v)\), and constructed \(\text{QLS}(4v)\text{s}\) and \(\text{QLS}(6v)\text{s}\) attaining all possible cardinalities for  \(v\geq2\). Subsequently, Zhang et al. \cite{Zhang6} completely solved the existence problem for \(\text{QLS}(v)\text{s}\) with \(v\geq8\) and \(v\notin\{9,11,23\}\), proving that every cardinality \(c\in[v,v^2]\setminus\{v+1\}\) is attainable. For further reading on quantum theory and quantum Latin squares, we recommend \cite{Claeys,Li,Mu3,Nielsen,Rather1,Rather2,Rather3,Reuttera,Zang1,Zang2,Zang3,Zang4,Zang5,Zyc}.

A {\it transversal} in a $\text{QLS}(v)$ is a set of $v$ elements, each located in a distinct row and a distinct column, that forms an orthonormal basis of $\mathcal{H}_v$. Two transversals in a $\text{QLS}(v)$ are called {\it disjoint} if they have no common cells.

\begin{definition}\label{def:t-resolvable}
Let \(t\) be a positive integer. A \(\text{QLS}(v)\) is called {\it \(t\)-resolvable}, denoted by \(t\text{-}\text{RQLS}(v)\), if it admits \(t\) sets of mutually disjoint transversals \(\mathcal{T}^{(r)}=\{T_i^{(r)}:i\in[v]\}\), \(r\in[t]\), such that \(|T_i^{(r)}\cap T_j^{(s)}|=1\) for all \(r,s\in[t]\) with
\(r\neq s\) and all \(i,j\in[v]\).
\end{definition}

Recently, Huang and Li \cite{Liyang} studied $\text{QLS}(v)$ with $v$ disjoint transversals which is actually \(1\text{-}\text{RQLS}(v)\), and pandiagonal quantum Latin square of odd order $v$ which is a
\(2\text{-}\text{RQLS}(v)\).

\begin{theorem} {\rm (\cite{Liyang})}\label{th:1.1}
 {\rm (1)} There exists a maximal-cardinality \(1\text{-}\text{RQLS}(v^2)\) for all  \(v\geq4\).\\
  {\rm (2)} There exists a maximal-cardinality \(1\text{-}\text{RQLS}(v)\) for any odd integer \(v\geq7\). \\
  {\rm (3)} For any \(v\notin\{5,7,11\}\) and \(\gcd(v,6)=1\) there exists  a maximal-cardinality \(2\text{-}\text{RQLS}(v)\).
  \end{theorem}

In this paper we almost solve the existence of a \(1\text{-}\text{RQLS}(v)\) with maximal cardinality based on Theorem~\ref{th:1.1} and obtain a class of \(2\text{-}\text{RQLS}(v)\) with maximal cardinality.

\begin{theorem}\label{th:1.2}
For any  integer  \(v\geq 7 \) and \( v \notin \{  10,14,18,22,26,30,34,38,42,46,54,58,62,\)\\\(66,74,82,94\} \),  there exists a \(1\text{-}\text{RQLS}(v)\)  with maximal cardinality.
\end{theorem}

\begin{theorem}\label{th:DRQLS}
For any integer \(n\geq6\), there exists a  \(2\text{-}\text{RQLS}(n^2)\) with maximal cardinality.
\end{theorem}

\section{\(t\text{-}\text{RQLS}(v)\) and \((t+1)\text{-}\text{MOQLS}(v)\)}

In this section, we establish the equivalence between \(t\text{-}\text{RQLS}(v)\text{s}\) and a set of \(t+1\) mutually orthogonal \(\text{QLS}(v)\text{s}\). Moreover, we provide an example of a maximal-cardinality \(1\text{-}\text{RQLS}(8)\).

\begin{definition}
Two $\text{QLS}(v)\text{s}$  $A=(|a_{i,j}\rangle)$ and $B=(|b_{i,j}\rangle)$ are orthogonal if the set $\{ |a_{i,j}\rangle\otimes |b_{i,j}\rangle : i,j \in [v]\}$  forms an orthonormal basis of $\mathcal{H}_v\otimes\mathcal{H}_v$, i.e.,
\[
\left( |a_{i,j}\rangle\otimes |b_{i,j}\rangle, |a_{i',j'}\rangle\otimes |b_{i',j'}\rangle \right)
=
(|a_{i,j}\rangle, |a_{i',j'}\rangle)
(|b_{i,j}\rangle, |b_{i',j'}\rangle)
=
\delta_{i,i'}\delta_{j,j'},
\]
for \(i,j,i',j'\in [v]\).
\end{definition}

A set of \(t\) mutually orthogonal \(\text{QLS}(v)\text{s}\) is denoted by \(t\text{-}\text{MOQLS}(v)\). If at least one of them is non-classical, then the \(t\text{-}\text{MOQLS}(v)\) is called non-classical.

\begin{lemma}\label{lem:t-resolvable-MOQLS}
\(A\) is a \(t\text{-}\text{RQLS}(v)\) if and only if there exist \(t\) classical \(\text{QLS}(v)\text{s}\), \(B_0,\ldots,B_{t-1}\), such that \(A,B_0,\ldots,B_{t-1}\) form a \((t+1)\text{-}\text{MOQLS}(v)\).
\end{lemma}

\begin{proof}
Suppose first that \(A\) is a \(t\text{-}\text{RQLS}(v)\). Let \(\mathcal{T}^{(r)}=\{T_k^{(r)}:k\in[v]\}\), \(r\in[t]\), be the corresponding \(t\) sets of mutually disjoint transversals. For each \(r\in[t]\), define a classical square \(B_r=(|b_{i,j}^{(r)}\rangle)\) by \(|b_{i,j}^{(r)}\rangle=|k\rangle\) whenever
\((i,j)\in T_k^{(r)}\).

Since the \(v\) transversals in \(\mathcal{T}^{(r)}\) are mutually disjoint, they partition all \(v^2\) cells. Moreover, each \(T_k^{(r)}\) contains exactly one cell in every row and every column. Hence \(B_r\) is a classical \(\text{QLS}(v)\).

We next show that \(A,B_0,\ldots,B_{t-1}\) are mutually orthogonal. Fix \(r\in[t]\). For two distinct cells, if they belong to the same \(T_k^{(r)}\), then the corresponding vectors of \(A\) are orthogonal; if they belong to distinct transversals of \(\mathcal{T}^{(r)}\), then the corresponding vectors of \(B_r\) are orthogonal. Moreover, the tensor product corresponding to each cell is a unit vector. Hence \(A\) and \(B_r\) are orthogonal.

Let \(r,s\in[t]\) with \(r\neq s\). Since \(|T_k^{(r)}\cap T_l^{(s)}|=1\) for all \(k,l\in[v]\), for each \(k,l\in[v]\), there is exactly one cell \((i,j)\) such that \(|b_{i,j}^{(r)}\rangle=|k\rangle\) and \(|b_{i,j}^{(s)}\rangle=|l\rangle\). Thus
\[
\{|b_{i,j}^{(r)}\rangle\otimes|b_{i,j}^{(s)}\rangle:i,j\in[v]\}
=
\{|k\rangle\otimes|l\rangle:k, l\in[v]\},
\]
which is an orthonormal basis of \(\mathcal H_v\otimes\mathcal H_v\). Hence \(B_r\) and \(B_s\) are orthogonal. Therefore \(A,B_0,\ldots,B_{t-1}\) form a
\((t+1)\text{-}\text{MOQLS}(v)\).

Conversely, suppose that \(A,B_0,\ldots,B_{t-1}\) form a \((t+1)\text{-}\text{MOQLS}(v)\), where \(B_0,\ldots,B_{t-1}\) are classical.  For each \(r\in[t]\) and \(k\in[v]\), define \(T_k^{(r)}=\{(i,j):|b_{i,j}^{(r)}\rangle=|k\rangle\}\). Since \(B_r\) is classical, each \(T_k^{(r)}\) contains exactly one cell
in every row and every column, and the sets \(T_k^{(r)}\), \(k\in[v]\),  are mutually disjoint and partition all \(v^2\) cells.

Let \((i,j)\) and \((i',j')\) be two distinct cells in \(T_k^{(r)}\). Then \((|b_{i,j}^{(r)}\rangle,|b_{i',j'}^{(r)}\rangle)=1\). Since \(A\) and \(B_r\) are orthogonal, it follows that \((|a_{i,j}\rangle,|a_{i',j'}\rangle)=0\). Thus the \(v\) vectors of \(A\) corresponding to the cells of \(T_k^{(r)}\) form an orthonormal basis of \(\mathcal H_v\), and hence \(T_k^{(r)}\) is a transversal of \(A\). Therefore, for each fixed \(r\in[t]\), the sets \(\{T_k^{(r)}:k\in[v]\}\) form a set of \(v\) mutually disjoint transversals of \(A\).

Finally, let \(r,s\in[t]\) with \(r\neq s\). Since the classical \(\text{QLS}(v)\text{s}\) \(B_r\) and \(B_s\) are orthogonal, each ordered pair \((|k\rangle,|l\rangle)\), \(k, l\in[v]\), occurs in exactly one cell. Therefore \(|T_k^{(r)}\cap T_l^{(s)}|=1\) for all \(k,l\in[v]\). Hence \(A\) is a \(t\text{-}\text{RQLS}(v)\).
\end{proof}

It is known from \cite{Colbourn} that there exist two orthogonal Latin squares of order \(v\) for every positive integer \(v\notin\{2,6\}\), and three mutually orthogonal Latin squares of order \(v\) for every positive integer \(v\notin\{2,3,6,10\}\). Therefore, by Lemma~\ref{lem:t-resolvable-MOQLS}, there exists a classical \(1\text{-}\text{RQLS}(v)\) for every \(v\notin\{2,6\}\), and a classical \(2\text{-}\text{RQLS}(v)\) for every \(v\notin\{2,3,6,10\}\). In this paper we focus on non-classical mutually orthogonal quantum Latin squares. By Lemma~\ref{lem:t-resolvable-MOQLS} we have the following conclusion.

\begin{corollary}\label{cor:MOQLS}
If there exists a maximal-cardinality \(t\text{-}\text{RQLS}(v)\), then there exists a non-classical \((t+1)\text{-}\text{MOQLS}(v)\).
\end{corollary}

Non-classical mutually orthogonal quantum Latin squares have been studied from several perspectives. Han et al.~\cite{Han} developed PBD and filling-in-holes constructions from combinatorial design theory and obtained several existence results for non-classical \(2\)- and \(3\text{-}\text{MOQLS}(v)\text{s}\). More recently, Ball and Simoens \cite{Ball2} investigated the maximum size of a non-classical set of mutually orthogonal quantum Latin squares. They proved that every set of \((v-2)\text{-}\text{MOQLS}(v)\text{s}\) is necessarily classical and constructed large non-classical sets for prime-power orders.

For small orders, Paczos et al.~\cite{Paczos} proved that the only possible cardinalities of \(\text{QLS}(2)\) and \(\text{QLS}(3)\) are \(2\) and \(3\), respectively. Since a \(\text{QLS}(2)\) of cardinality \(2\) contains no transversal, no \(1\text{-}\text{RQLS}(2)\) exists, whereas every \(1\text{-}\text{RQLS}(3)\) has cardinality \(3\). Ball and Simoens \cite{Ball} proved that \(2\text{-}\text{MOQLS}(v)\text{s}\) are necessarily classical for \(v\in\{4,5\}\), whereas no \(2\text{-}\text{MOQLS}(6)\text{s}\) exist. Therefore, by Lemma~\ref{lem:t-resolvable-MOQLS}, every \(1\text{-}\text{RQLS}(v)\) is classical for \(v\in\{4,5\}\) and hence has cardinality \(v\), while no \(1\text{-}\text{RQLS}(6)\) exists.

\begin{example} \label{ex:3.3}
We present a $1\text{-}\text{RQLS}(8)$ with maximal cardinality $c=64$, where $\omega=e^{\frac{\pi \mathrm{i}}{4}}$ is a primitive $8$-th root of unity.

\begin{center}

\scriptsize{
\begin{tabular}{|>{\centering\arraybackslash}*{8}{>{\centering\arraybackslash}m{1.4cm}|}}
\hline
\qcell{\ket{0}} & \qcell{\ket{1}} & \qcell{\ket{2}} & \qcell{\ket{3}} & \qcell{\ket{4}} & \qcell{\ket{5}} & \qcell{\ket{6}} & \qcell{\ket{7}} \\ \hline

\qcell{\tfrac{1}{\sqrt2}(\ket{2}\\+\omega^{3}\ket{3})} &
\qcell{\tfrac{1}{\sqrt2}(\ket{2}\\+\omega^{7}\ket{3})} &
\qcell{\tfrac{1}{\sqrt2}(\ket{0}\\+\omega^{6}\ket{1})} &
\qcell{\tfrac{1}{\sqrt2}(\ket{0}\\+\omega^{2}\ket{1})} &
\qcell{\tfrac{1}{\sqrt2}(\ket{6}\\+\omega^{4}\ket{7})} &
\qcell{\tfrac{1}{\sqrt2}(\ket{6}\\+\ket{7})} &
\qcell{\tfrac{1}{\sqrt2}(\ket{4}\\+\omega^{6}\ket{5})} &
\qcell{\tfrac{1}{\sqrt2}(\ket{4}\\+\omega^{2}\ket{5})} \\ \hline

\qcell{\tfrac12(\ket{4}\\ +\ket{5}\\+\omega^{7}\ket{6}\\+\omega^{5}\ket{7})} &
\qcell{\tfrac12(\ket{4}\\+\ket{5}\\ +\omega^{3}\ket{6}\\+\omega\ket{7})} &
\qcell{\tfrac12(\ket{4}\\+\omega^{4}\ket{5}\\ +\omega^{3}\ket{6}\\+\omega^{5}\ket{7})} &
\qcell{\tfrac12(\ket{4}\\+\omega^{4}\ket{5}\\ +\omega^{7}\ket{6}\\+\omega\ket{7})} &
\qcell{\tfrac12(\ket{0}\\+\ket{1}\\ +\omega^{2}\ket{2}\\+\omega^{7}\ket{3})} &
\qcell{\tfrac12(\ket{0}\\+\ket{1}\\ +\omega^{6}\ket{2}\\+\omega^{3}\ket{3})} &
\qcell{\tfrac12(\ket{0}\\+\omega^{4}\ket{1}\\ +\omega^{6}\ket{2}\\+\omega^{7}\ket{3})} &
\qcell{\tfrac12(\ket{0}\\+\omega^{4}\ket{1}\\ +\omega^{2}\ket{2}\\+\omega^{3}\ket{3})} \\ \hline

\qcell{\tfrac12(\ket{4}\\+\omega^{4}\ket{5}\\ +\omega\ket{6}\\+\omega^{3}\ket{7})} &
\qcell{\tfrac12(\ket{4}\\+\omega^{4}\ket{5}\\ +\omega^{5}\ket{6}\\+\omega^{7}\ket{7})} &
\qcell{\tfrac12(\ket{4}\\+\ket{5}\\ +\omega^{5}\ket{6}\\+\omega^{3}\ket{7})} &
\qcell{\tfrac12(\ket{4}\\+\ket{5}\\ +\omega\ket{6}\\+\omega^{7}\ket{7})} &
\qcell{\tfrac12(\ket{0}\\+\omega^{4}\ket{1}\\ +\omega^{4}\ket{2}\\+\omega^{5}\ket{3})} &
\qcell{\tfrac12(\ket{0}\\+\omega^{4}\ket{1}\\ +\ket{2}\\+\omega\ket{3})} &
\qcell{\tfrac12(\ket{0}\\+\ket{1}\\ +\ket{2}\\+\omega^{5}\ket{3})} &
\qcell{\tfrac12(\ket{0}\\+\ket{1}\\ +\omega^{4}\ket{2}\\+\omega\ket{3})} \\ \hline

\qcell{\tfrac12(\sqrt2\ket{1}\\+\omega^{7}\ket{2}\\ +\omega^{6}\ket{3})} &
\qcell{\tfrac12(\sqrt2\ket{0}\\+\omega^{5}\ket{2}\\ +\ket{3})} &
\qcell{\tfrac12(\ket{0}\\+\omega^{2}\ket{1}\\ +\sqrt2\omega^{4}\ket{3})} &
\qcell{\tfrac12(\ket{0}\\+\omega^{6}\ket{1}\\ +\sqrt2\omega\ket{2})} &
\qcell{\tfrac12(\sqrt2\ket{5}\\+\omega^{4}\ket{6}\\ +\omega^{4}\ket{7})} &
\qcell{\tfrac12(\sqrt2\ket{4}\\+\omega^{2}\ket{6}\\ +\omega^{6}\ket{7})} &
\qcell{\tfrac12(\ket{4}\\+\omega^{2}\ket{5}\\ +\sqrt2\omega^{2}\ket{7})} &
\qcell{\tfrac12(\ket{4}\\+\omega^{6}\ket{5}\\ +\sqrt2\omega^{6}\ket{6})} \\ \hline

\qcell{\tfrac12(\sqrt2\ket{1}\\+\omega^{3}\ket{2}\\ +\omega^{2}\ket{3})} &
\qcell{\tfrac12(\sqrt2\ket{0}\\+\omega\ket{2}\\ +\omega^{4}\ket{3})} &
\qcell{\tfrac12(\ket{0}\\+\omega^{2}\ket{1}\\ +\sqrt2\ket{3})} &
\qcell{\tfrac12(\ket{0}\\+\omega^{6}\ket{1}\\ +\sqrt2\omega^{5}\ket{2})} &
\qcell{\tfrac12(\sqrt2\ket{5}\\+\ket{6}\\ +\ket{7})} &
\qcell{\tfrac12(\sqrt2\ket{4}\\+\omega^{6}\ket{6}\\ +\omega^{2}\ket{7})} &
\qcell{\tfrac12(\ket{4}\\+\omega^{2}\ket{5}\\ +\sqrt2\omega^{6}\ket{7})} &
\qcell{\tfrac12(\ket{4}\\+\omega^{6}\ket{5}\\ +\sqrt2\omega^{2}\ket{6})} \\ \hline

\qcell{\tfrac12(\ket{4}\\+\omega^{6}\ket{5}\\ +\sqrt2\ket{7})} &
\qcell{\tfrac12(\ket{4}\\+\omega^{2}\ket{5}\\ +\sqrt2\ket{6})} &
\qcell{\tfrac12(\sqrt2\ket{5}\\+\omega^{2}\ket{6}\\ +\omega^{6}\ket{7})} &
\qcell{\tfrac12(\sqrt2\ket{4}\\+\omega^{4}\ket{6}\\ +\omega^{4}\ket{7})} &
\qcell{\tfrac12(\ket{0}\\+\omega^{6}\ket{1}\\ +\sqrt2\omega^{2}\ket{3})} &
\qcell{\tfrac12(\ket{0}\\+\omega^{2}\ket{1}\\ +\sqrt2\omega^{3}\ket{2})} &
\qcell{\tfrac12(\sqrt2\ket{1}\\+\omega^{5}\ket{2}\\ +\ket{3})} &
\qcell{\tfrac12(\sqrt2\ket{0}\\+\omega^{7}\ket{2}\\ +\omega^{6}\ket{3})} \\ \hline

\qcell{\tfrac12(\ket{4}\\+\omega^{2}\ket{5}\\ +\sqrt2\omega^{4}\ket{6})} &
\qcell{\tfrac12(\ket{4}\\+\omega^{6}\ket{5}\\ +\sqrt2\omega^{4}\ket{7})} &
\qcell{\tfrac12(\sqrt2\ket{4}\\+\ket{6}\\ +\ket{7})} &
\qcell{\tfrac12(\sqrt2\ket{5}\\+\omega^{6}\ket{6}\\ +\omega^{2}\ket{7})} &
\qcell{\tfrac12(\ket{0}\\+\omega^{2}\ket{1}\\ +\sqrt2\omega^{7}\ket{2})} &
\qcell{\tfrac12(\ket{0}\\+\omega^{6}\ket{1}\\ +\sqrt2\omega^{6}\ket{3})} &
\qcell{\tfrac12(\sqrt2\ket{0}\\+\omega^{3}\ket{2}\\ +\omega^{2}\ket{3})} &
\qcell{\tfrac12(\sqrt2\ket{1}\\+\omega\ket{2}\\ +\omega^{4}\ket{3})} \\ \hline
\end{tabular}}
\end{center}
The following are eight mutually disjoint transversals of this \(1\text{-}\text{RQLS}(8)\):
\[
\begin{aligned}
T_0={}&\{(0,0),(1,1),(2,2),(3,3),(4,4),(5,5),(6,6),(7,7)\},\\
T_1={}&\{(0,1),(1,0),(2,3),(3,2),(4,5),(5,4),(6,7),(7,6)\},\\
T_2={}&\{(0,2),(1,3),(2,0),(3,1),(4,6),(5,7),(6,4),(7,5)\},\\
T_3={}&\{(0,3),(1,2),(2,1),(3,0),(4,7),(5,6),(6,5),(7,4)\},\\
T_4={}&\{(0,4),(1,5),(2,6),(3,7),(4,0),(5,1),(6,2),(7,3)\},\\
T_5={}&\{(0,5),(1,4),(2,7),(3,6),(4,1),(5,0),(6,3),(7,2)\},\\
T_6={}&\{(0,6),(1,7),(2,4),(3,5),(4,2),(5,3),(6,0),(7,1)\},\\
T_7={}&\{(0,7),(1,6),(2,5),(3,4),(4,3),(5,2),(6,1),(7,0)\}.
\end{aligned}
\]
\end{example}

\section{Complete Mapping  Constructions}

Complete mappings are classical tools in combinatorial design theory and have been extensively used in the construction of Latin squares and related combinatorial structures \cite{Colbourn}.  Recently, Huang and Li \cite{Liyang} introduced complete mapping techniques into the study of quantum Latin squares. Ball and Simoens \cite{Ball2} extended this line of work by developing a corresponding construction over Frobenius rings. In this section, we employ complete mappings to construct maximal-cardinality
\(1\text{-}\text{RQLS}(v)\text{s}\) and \(2\text{-}\text{RQLS}(v)\text{s}\).

\begin{definition}
Let $(G,+)$ be a finite group, and let $\operatorname{id}$ denote the identity permutation of $G$. A permutation $\mu$ of $G$ is called a complete mapping of $G$ if $\operatorname{id}+\mu$ is also a permutation of $G$, where
$(\operatorname{id}+\mu)(t)=t+\mu(t)$ for any $t\in G$.
\end{definition}

\begin{definition}
Let \((G,+)\) be a finite abelian group. A {\it character} of \(G\) is a map \(\chi:G\rightarrow\mathbb C\) satisfying \(|\chi(t)|=1\) for all \(t\in G\) and
$\chi(t+s)=\chi(t)\chi(s)$ for all  $t,s\in G$.
\end{definition}

\subsection{Constructions of \(1\text{-}\text{RQLS}(v)\text{s}\)}

Let \(v=2m\) and set \(G=\mathbb{Z}_2\times\mathbb{Z}_m\). Let \(\omega=e^{\frac{2\pi \mathrm{i}}{m}}\) be a primitive \(m\)-th root of unity. For each \((a,b)\in G\), define \(\chi_{a,b}:G\rightarrow \mathbb{C}\) by
\[
\chi_{a,b}(t_1,t_2)=(-1)^{a t_1}\omega^{b t_2}, \qquad (t_1,t_2)\in G.
\]
Then \(\chi_{a,b}\), \((a,b)\in G\), are precisely all the characters of \(G\). Moreover, for any \((a,b),(a',b')\in G\),
\[
\overline{\chi_{a,b}(t_1,t_2)}\chi_{a',b'}(t_1,t_2)=\chi_{a'-a,\;b'-b}(t_1,t_2), \qquad (t_1,t_2)\in G.
\]
In particular,
\[
\sum_{(t_1,t_2)\in G}\chi_{a,b}(t_1,t_2)=
\begin{cases}
2m, & \text{if } (a,b)=(0,0),\\
0, & \text{otherwise}.
\end{cases}
\]

For each \(t\in[v]\), write uniquely \(t=mt_1+t_2\), where \(t_1\in\{0,1\}\) and \(t_2\in\{0,1,\ldots,m-1\}\), and identify \(t\) with the corresponding element \((t_1,t_2)\in G\).  Thus, the rows and columns of the array, as well as the computational basis of \(\mathcal H_v\), are indexed by the elements of \(G\).

\begin{construction}\label{con:Z2Zm} Let \(v=2m\) and $G=\mathbb{Z}_2\times\mathbb{Z}_m$. Let $\mu:G\rightarrow G$ be a complete mapping of $G$. Define an array $$L=\bigl(|L_{(i_1,i_2),(j_1,j_2)}\rangle\bigr)_{(i_1,i_2),(j_1,j_2)\in G}$$ by
\[
|L_{(i_1,i_2),(j_1,j_2)}\rangle
=
\frac{1}{\sqrt{2m}}
\sum_{(t_1,t_2)\in G}
\chi_{i_1,i_2}\bigl(\mu(t_1,t_2)\bigr)
\chi_{j_1,j_2}(t_1,t_2)
|(t_1,t_2)\rangle.
\]
Then $L$ is a $1\text{-}\text{RQLS}(v)$.
\end{construction}

\begin{proof}
We first prove that \(L\) is a \(\text{QLS}(v)\).  For a fixed row \((i_1,i_2)\in G\) and any two columns \((j_1,j_2),(j'_1,j'_2)\in G\), we have
\[
\begin{aligned}
&\left( |L_{(i_1,i_2),(j_1,j_2)}\rangle~,~|L_{(i_1,i_2),(j'_1,j'_2)}\rangle \right)  \\
&=
\frac{1}{2m}\sum_{(t_1,t_2)\in G}
\overline{
\chi_{i_1,i_2}\bigl(\mu(t_1,t_2)\bigr)
\chi_{j_1,j_2}(t_1,t_2)
}
\chi_{i_1,i_2}\bigl(\mu(t_1,t_2)\bigr)
\chi_{j'_1,j'_2}(t_1,t_2) \\
&=
\frac{1}{2m}\sum_{(t_1,t_2)\in G}
\overline{\chi_{j_1,j_2}(t_1,t_2)}
\chi_{j'_1,j'_2}(t_1,t_2) \\
&=
\frac{1}{2m}\sum_{(t_1,t_2)\in G}
\chi_{j'_1-j_1,\;j'_2-j_2}(t_1,t_2).
\end{aligned}
\]
By the character-sum property, the last expression is  \(1\) when \((j_1,j_2)=(j'_1,j'_2)\), and \(0\) otherwise. Hence every row of \(L\) is an orthonormal basis of \(\mathcal H_v\).

Similarly, for a fixed column \((j_1,j_2)\in G\) and any two rows \((i_1,i_2),(i'_1,i'_2)\in G\),
\[
\begin{aligned}
&\left( |L_{(i_1,i_2),(j_1,j_2)}\rangle,
|L_{(i'_1,i'_2),(j_1,j_2)}\rangle \right)  \\
&=
\frac{1}{2m}\sum_{(t_1,t_2)\in G}
\overline{
\chi_{i_1,i_2}\bigl(\mu(t_1,t_2)\bigr)
\chi_{j_1,j_2}(t_1,t_2)
}
\chi_{i'_1,i'_2}\bigl(\mu(t_1,t_2)\bigr)
\chi_{j_1,j_2}(t_1,t_2) \\
&=
\frac{1}{2m}\sum_{(t_1,t_2)\in G}
\overline{\chi_{i_1,i_2}\bigl(\mu(t_1,t_2)\bigr)}
\chi_{i'_1,i'_2}\bigl(\mu(t_1,t_2)\bigr) \\
&=
\frac{1}{2m}\sum_{(t_1,t_2)\in G}
\chi_{i'_1-i_1,\;i'_2-i_2}
\bigl(\mu(t_1,t_2)\bigr).
\end{aligned}
\]
Since \(\mu\) is a permutation of \(G\), the character-sum property implies that the above inner product is   \(1\) when \((i_1,i_2)=(i'_1,i'_2)\), and \(0\) otherwise. Hence every column of \(L\) is an orthonormal basis of \(\mathcal H_v\). Therefore \(L\) is a \(\text{QLS}(v)\).

It remains to prove that \(L\) is \(1\)-resolvable. For each \((c_1,c_2)\in G\), define
\[
T_{(c_1,c_2)}=\left\{\bigl((i_1,i_2),(i_1+c_1,i_2+c_2)\bigr):(i_1,i_2)\in G\right\},
\]
where the additions are taken in \(G=\mathbb Z_2\times\mathbb Z_m\). Each \(T_{(c_1,c_2)}\) contains exactly one cell in every row and every column. Moreover, the sets \(T_{(c_1,c_2)}\), \((c_1,c_2)\in G\), are mutually disjoint and partition all cells of \(L\).

We next show that each \(T_{(c_1,c_2)}\) is a transversal. Fix \((c_1,c_2)\in G\). For \((i_1,i_2),(i'_1,i'_2)\in G\),
\[
\begin{aligned}
&\left( |L_{(i_1,i_2),(i_1+c_1,i_2+c_2)}\rangle~,~ |L_{(i'_1,i'_2),(i'_1+c_1,i'_2+c_2)}\rangle \right)  \\
&=
\frac{1}{2m}\sum_{(t_1,t_2)\in G}
\chi_{i'_1-i_1,\;i'_2-i_2}
\bigl(\mu(t_1,t_2)\bigr)
\chi_{i'_1-i_1,\;i'_2-i_2}(t_1,t_2) \\
&=
\frac{1}{2m}
\sum_{(t_1,t_2)\in G}
\chi_{i'_1-i_1,\;i'_2-i_2}
\bigl(\mu(t_1,t_2)+(t_1,t_2)\bigr),
\end{aligned}
\]
where the addition is taken in \(G\). Since \(\mu\) is a complete mapping of \(G\), \(\mu(t_1,t_2)+(t_1,t_2)\) also ranges over all elements of \(G\) as \((t_1,t_2)\) ranges over \(G\). Then the above inner product is  \(1\) when \((i_1,i_2)=(i'_1,i'_2)\), and  \(0\) otherwise. Therefore the vectors corresponding to the cells in \(T_{(c_1,c_2)}\) form an orthonormal basis of \(\mathcal H_v\), and hence \(T_{(c_1,c_2)}\) is a transversal.

Thus $\{T_{(c_1,c_2)}:(c_1,c_2)\in G\}$ is a set of \(v=|G|=2m\) mutually disjoint transversals of \(L\). Therefore \(L\) is a \(1\text{-}\text{RQLS}(v)\).
\end{proof}

\begin{lemma}\label{lem:max-card-Z2Zm}
Let \(L\) be the array defined in Construction~\ref{con:Z2Zm}. Suppose that \(\mu\) satisfies the following condition: for any
\((a_1,a_2),(b_1,b_2)\in G\),
$
\chi_{a_1,a_2}
\bigl(\mu(t_1,t_2)\bigr)
\chi_{b_1,b_2}(t_1,t_2)
$
is constant on \(G\) only when \((a_1,a_2)=(b_1,b_2)=(0,0)\). Then \(L\) has maximal cardinality \(v^2\).
\end{lemma}

\begin{proof}
Suppose that \(|L_{(i_1,i_2),(j_1,j_2)}\rangle=e^{\mathrm{i}\theta}|L_{(i'_1,i'_2),(j'_1,j'_2)}\rangle\) for some \(\theta\in\mathbb R\). Let
\((a_1,a_2)=(i_1-i'_1,i_2-i'_2)\) and   \((b_1,b_2)=(j_1-j'_1,j_2-j'_2)\). By the definition of \(L\), comparing the coefficients of \(|(t_1,t_2)\rangle\) gives
$
\chi_{a_1,a_2}\bigl(\mu(t_1,t_2)\bigr)
\chi_{b_1,b_2}(t_1,t_2)
=
e^{\mathrm{i}\theta}
$
for all \((t_1,t_2)\in G\). Hence \(\chi_{a_1,a_2}\bigl(\mu(t_1,t_2)\bigr)\chi_{b_1,b_2}(t_1,t_2)\) is constant on \(G\). By the assumption on \(\mu\), we have $(a_1,a_2)=(b_1,b_2)=(0,0)$.
Therefore $(i_1,i_2)=(i'_1,i'_2)$  and $(j_1,j_2)=(j'_1,j'_2)$. Thus identical entries of \(L\) must occur in the same cell. Hence all \(v^2\) entries of \(L\) are distinct, and \(L\) has maximal cardinality \(v^2\).
\end{proof}

\begin{lemma}\label{lem:complete-mapping-Z2Zm}
 Let \(m=2k\geq6\) and \(G=\mathbb Z_2\times\mathbb Z_m\). Define \(\mu : G\rightarrow G\) by
\[
\mu(0,x)=
\begin{cases}
(1,0), & x=0,\\
(0,x), & 1\leq x\leq k-1,\\
(1,x+1), & k\leq x\leq 2k-2,\\
(0,0), & x=2k-1,
\end{cases}
\]
and
\[
\mu(1,x)=
\begin{cases}
(1,x+1), & 0\leq x\leq k-2,\\
(0,k), & x=k-1,\\
(1,k), & x=k,\\
(0,x), & k+1\leq x\leq 2k-1.
\end{cases}
\]
Then \(\mu\) is a complete mapping of \(G\) and satisfies the condition in Lemma~\ref{lem:max-card-Z2Zm}.
\end{lemma}

\begin{proof}
We first show that \(\mu\) is a complete mapping of \(G\). By the definition of \(\mu\),
\[
\mu(\{0\}\times\mathbb Z_m)= \{(0,x):0\leq x\leq k-1\}  \cup\{(1,0)\}\cup\{(1,x):k+1\leq x\leq 2k-1\},
\]
and
\[
\mu(\{1\}\times\mathbb Z_m)
= \{(0,x):k\leq x\leq 2k-1\} \cup\{(1,x):1\leq x\leq k\}.
\]
These two sets are disjoint and their union is \(G\). Hence \(\mu\) is a permutation of \(G\).

Moreover,
\[
(\operatorname{id}+\mu)(0,x)=
\begin{cases}
(1,0), & x=0,\\
(0,2x), & 1\leq x\leq k-1,\\
(1,2x+1), & k\leq x\leq 2k-2,\\
(0,2k-1), & x=2k-1,
\end{cases}
\]
and
\[
(\operatorname{id}+\mu)(1,x)=
\begin{cases}
(0,2x+1), & 0\leq x\leq k-2,\\
(1,2k-1), & x=k-1,\\
(0,0), & x=k,\\
(1,2x), & k+1\leq x\leq 2k-1,
\end{cases}
\]
where the second components are taken modulo \(m=2k\). It follows that
\[
\begin{aligned}
(\operatorname{id}+\mu)(\{0\}\times\mathbb Z_m)
={}&
\{(0,2),(0,4),\ldots,(0,2k-2),(0,2k-1)\}\\
&\cup
\{(1,0),(1,1),(1,3),\ldots,(1,2k-3)\},
\end{aligned}
\]
while
\[
\begin{aligned}
(\operatorname{id}+\mu)(\{1\}\times\mathbb Z_m)
={}&
\{(0,0),(0,1),(0,3),\ldots,(0,2k-3)\}\\
&\cup
\{(1,2),(1,4),\ldots,(1,2k-2),(1,2k-1)\}.
\end{aligned}
\]
These two sets are disjoint and their union is \(G\). Thus \(\operatorname{id}+\mu\) is a permutation of \(G\), and hence \(\mu\) is a complete mapping of \(G\).

It remains to verify the condition in Lemma~\ref{lem:max-card-Z2Zm}. Let \((a_1,a_2),(b_1,b_2)\in G\), and define
\[
F(t_1,t_2)
=
\chi_{a_1,a_2}\bigl(\mu(t_1,t_2)\bigr)
\chi_{b_1,b_2}(t_1,t_2).
\]
Suppose that \(F\) is constant on \(G\). Since \(k\geq3\), we have \(\mu(0,1)=(0,1)\) and \(\mu(0,2)=(0,2)\). Thus \(F(0,1)=\omega^{a_2+b_2}\) and \(F(0,2)=\omega^{2(a_2+b_2)}\). Since \(F(0,1)=F(0,2)\), it follows that \(\omega^{a_2+b_2}=1\),   and hence \(a_2+b_2\equiv0\pmod m\). In particular, \(F(0,1)=1\), so the constant value of \(F\) is \(1\).

Since \(\mu(0,0)=(1,0)\), we have \(F(0,0)=(-1)^{a_1}\), and hence \(a_1=0\). Moreover, \(\mu(1,k+1)=(0,k+1)\), so \(F(1,k+1)=(-1)^{b_1}\omega^{(a_2+b_2)(k+1)}=(-1)^{b_1}\), which yields \(b_1=0\).

Finally, since \(\mu(0,k)=(1,k+1)\), \(F(0,k)=(-1)^{a_1}\omega^{a_2(k+1)+b_2k}=\omega^{a_2}\), where \(a_1=0\) and \(a_2+b_2=0\) have been used. Since the constant value of \(F\) is \(1\), we obtain \(a_2=0\), and consequently \(b_2=0\). Therefore
$
(a_1,a_2)=(b_1,b_2)=(0,0).
$
Thus \(\mu\) satisfies the condition in Lemma~\ref{lem:max-card-Z2Zm}.
\end{proof}

\begin{lemma}\label{th:max-RQLS-4k}
For every integer \(v\geq8\) with \(v\equiv0\pmod4\), there exists a maximal-cardinality \(1\text{-}\text{RQLS}(v)\) whose main diagonal and anti-diagonal are both transversals.
\end{lemma}

\begin{proof}
For \(v=8\), the result follows from Example~\ref{ex:3.3}, since \(T_0\) and \(T_7\) are the main diagonal and anti-diagonal, respectively.

Let \(v=4k\geq12\). By Construction~\ref{con:Z2Zm}, Lemma~\ref{lem:max-card-Z2Zm}, and Lemma~\ref{lem:complete-mapping-Z2Zm}, with \(m=2k\), there exists a
maximal-cardinality \(1\text{-}\text{RQLS}(v)\). By the proof of Construction~\ref{con:Z2Zm}, both
\[
T_{(0,0)}
=
\left\{
\bigl((i_1,i_2),(i_1,i_2)\bigr):(i_1,i_2)\in G
\right\}
\]
and
\[
T_{(1,0)}
=
\left\{
\bigl((i_1,i_2),(i_1+1,i_2)\bigr):(i_1,i_2)\in G
\right\}
\]
are transversals. Relabel the rows and columns by the elements of
\([v]=\{0,1,\ldots,4k-1\}\) according to the following ordering of \(G\):
\[
(0,0),(0,1),\ldots,(0,2k-1),
(1,2k-1),(1,2k-2),\ldots,(1,0).
\]

For \(T_{(0,0)}\), the cell \(\bigl((0,i_2),(0,i_2)\bigr)\) is relabeled as the cell in row \(i_2\) and column \(i_2\), while
\(\bigl((1,i_2),(1,i_2)\bigr)\) is relabeled as the cell in row \(4k-1-i_2\) and column \(4k-1-i_2\), for \(0\leq i_2\leq2k-1\). Hence the \(4k\) cells of \(T_{(0,0)}\) are precisely the cells on the main diagonal.

For \(T_{(1,0)}\), the cell \(\bigl((0,i_2),(1,i_2)\bigr)\) is relabeled as the cell in row \(i_2\) and column \(4k-1-i_2\), while \(\bigl((1,i_2),(0,i_2)\bigr)\) is relabeled as the cell in row \(4k-1-i_2\) and column \(i_2\), for \(0\leq i_2\leq2k-1\). Hence the \(4k\) cells of \(T_{(1,0)}\) are precisely the cells on the anti-diagonal.

Since row and column permutations preserve orthonormality, resolvability, and cardinality, the resulting square remains a maximal-cardinality \(1\text{-}\text{RQLS}(v)\) whose main diagonal and anti-diagonal are both transversals.
\end{proof}

\subsection{Constructions of \(2\text{-}\text{RQLS}(v)\text{s}\)}

We next apply the complete mapping construction to obtain maximal-cardinality \(2\text{-}\text{RQLS}(v)\text{s}\), thereby proving
Theorem~\ref{th:DRQLS}.

\textbf{Proof of Theorem~\ref{th:DRQLS}.}  Let \(G=\mathbb Z_n\times\mathbb Z_n\) and \(\zeta=e^{\frac{2\pi\mathrm{i}}{n}}\). For \((a,b)\in G\), let \(\chi_{a,b}\) be the character of \(G\) defined by
\(\chi_{a,b}(t_1,t_2)=\zeta^{at_1+bt_2}\), \((t_1,t_2)\in G\). Then
\[
\sum_{(t_1,t_2)\in G}\chi_{a,b}(t_1,t_2)
=
\begin{cases}
n^2, & \text{if } (a,b)=(0,0),\\
0, & \text{otherwise}.
\end{cases}
\]

Let \(g=(n-2\ \ n-1)\) and \(f=(2\ \ 3)(n-2\ \ n-1)\) be permutations of \(\mathbb Z_n\). Since \(n\geq6\), the transpositions \((2\ \ 3)\) and \((n-2\ \ n-1)\) are disjoint. From the definitions of \(f\) and \(g\), it follows that \(t+f(t)-g(t)=(2\ \ 3)(t)\) for every \(t\in\mathbb Z_n\). Hence \(t+f(t)-g(t)\) defines a permutation of \(\mathbb Z_n\). Define
\[
\mu(t_1,t_2)=\bigl(g(t_2),\,t_1+f(t_2)\bigr).
\]
Since \(g\) is a permutation of \(\mathbb Z_n\), the first coordinate of \(\mu(t_1,t_2)\) uniquely determines \(t_2\), and then the second coordinate
uniquely determines \(t_1\). Hence \(\mu\) is a permutation of \(G\).
Moreover,
\[
(\operatorname{id}+\mu)(t_1,t_2)=\bigl(t_1+g(t_2),\,t_1+t_2+f(t_2)\bigr),
\]
where the addition is taken in \(G=\mathbb Z_n\times\mathbb Z_n\).  Suppose that \((\operatorname{id}+\mu)(t_1,t_2)=(\operatorname{id}+\mu)(t'_1,t'_2)\).
Then \(t_1+g(t_2)=t'_1+g(t'_2)\) and \(t_1+t_2+f(t_2)=t'_1+t'_2+f(t'_2)\). Subtracting the first equality from the second gives
\[
t_2+f(t_2)-g(t_2)
=
t'_2+f(t'_2)-g(t'_2).
\]
Since \(t+f(t)-g(t)=(2\ \ 3)(t)\) defines  a permutation of \(\mathbb Z_n\), we have \(t_2=t'_2\), and hence \(t_1=t'_1\). Therefore, \(\operatorname{id}+\mu\) is injective and thus a permutation of \(G\). Hence \(\mu\) is a complete mapping.

Define an array
\(
L=\bigl(|L_{(i_1,i_2),(j_1,j_2)}\rangle\bigr)_{(i_1,i_2),(j_1,j_2)\in G}
\)
by
\[
|L_{(i_1,i_2),(j_1,j_2)}\rangle
=
\frac{1}{n}
\sum_{(t_1,t_2)\in G}
\chi_{i_1,i_2}\bigl(\mu(t_1,t_2)\bigr)
\chi_{j_1,j_2}(t_1,t_2)
|(t_1,t_2)\rangle.
\]
Since \(\mu\) is a complete mapping of \(G\), by the character-sum property above and the same argument as in the proof of Construction~\ref{con:Z2Zm}, \(L\) is a \(\text{QLS}(n^2)\),   and
\[
T_{(c_1,c_2)}
=
\left\{
\bigl((i_1,i_2),(i_1+c_1,i_2+c_2)\bigr):(i_1,i_2)\in G
\right\},
\qquad (c_1,c_2)\in G,
\]
form a set of \(n^2\) mutually disjoint transversals of \(L\).

We next verify that \(L\) has maximal cardinality. For
\((a_1,a_2),(b_1,b_2)\in G\), define
\[
F(t_1,t_2)
=
\chi_{a_1,a_2}\bigl(\mu(t_1,t_2)\bigr)
\chi_{b_1,b_2}(t_1,t_2).
\]
Suppose that \(F\) is constant on \(G\). Since
\(\mu(0,0)=(0,0)\), we have \(F(0,0)=1\), and hence
\(F(t_1,t_2)=1\) for all \((t_1,t_2)\in G\). Moreover,
\[
\mu(1,0)=(0,1),\qquad
\mu(0,1)=(1,1),\qquad
\mu(0,2)=(2,3).
\]
Evaluating \(F\) at \((1,0)\), \((0,1)\), and \((0,2)\), respectively,
yields
\[
a_2+b_1\equiv0,\qquad
a_1+a_2+b_2\equiv0,\qquad
2a_1+3a_2+2b_2\equiv0
\pmod n.
\]
Subtracting twice the second congruence from the third gives \(a_2=0\). It follows from the first two congruences that \(b_1=0\) and \(b_2=-a_1\). Finally, since \(\mu(0,n-2)=(n-1,n-1)\), the equality \(F(0,n-2)=1\) gives \(a_1=0\), and hence \(b_2=0\). Therefore, \((a_1,a_2)=(b_1,b_2)=(0,0)\). By the same argument as in Lemma~\ref{lem:max-card-Z2Zm}, \(L\) has maximal cardinality \(n^4\).

It remains to construct a second set of \(n^2\) mutually disjoint transversals. Define \(\sigma(t_1,t_2)=(t_1-t_2,-t_1)\). Clearly, \(\sigma\) is a permutation of \(G\). Moreover,
$
(\sigma-\operatorname{id})(t_1,t_2)=(-t_2,-t_1-t_2).
$
If \((\sigma-\operatorname{id})(t_1,t_2) =(\sigma-\operatorname{id})(t'_1,t'_2)\), then the first coordinates give \(t_2=t'_2\), and the second coordinates give \(t_1=t'_1\). Hence \(\sigma-\operatorname{id}\) is also a permutation of \(G\). Moreover,
$
(\mu+\sigma)(t_1,t_2)
=
\bigl(t_1+g(t_2)-t_2,\,f(t_2)\bigr),
$
which is a permutation of \(G\), since its second component uniquely determines \(t_2\), and then its first component uniquely determines \(t_1\).

For each \((d_1,d_2)\in G\), define
\[
S_{(d_1,d_2)}
=
\left\{
\bigl((i_1,i_2),\sigma(i_1,i_2)+(d_1,d_2)\bigr):
(i_1,i_2)\in G
\right\}.
\]
Since \(\sigma\) is a permutation, the sets \(S_{(d_1,d_2)}\), \((d_1,d_2)\in G\), are mutually disjoint and partition all cells of \(L\), and each contains exactly one cell in every row and every column. For any two cells \(\bigl((i_1,i_2),\sigma(i_1,i_2)+(d_1,d_2)\bigr)\) and
\(\bigl((i'_1,i'_2),\sigma(i'_1,i'_2)+(d_1,d_2)\bigr)\) in \(S_{(d_1,d_2)}\), the inner product of the corresponding vectors is
\[
\frac{1}{n^2}
\sum_{(t_1,t_2)\in G}
\chi_{i'_1-i_1,\;i'_2-i_2}\bigl(\mu(t_1,t_2)\bigr)
\chi_{\sigma(i'_1,i'_2)-\sigma(i_1,i_2)}(t_1,t_2).
\]
By the definition of \(\sigma\),
\(
\sigma(i'_1,i'_2)-\sigma(i_1,i_2)
=
\sigma(i'_1-i_1,i'_2-i_2).
\)
Hence, by the definition of the characters,
\[
\chi_{\sigma(i'_1,i'_2)-\sigma(i_1,i_2)}(t_1,t_2)
=
\chi_{i'_1-i_1,\;i'_2-i_2}\bigl(\sigma(t_1,t_2)\bigr).
\]
Therefore, the inner product of the two corresponding vectors can be rewritten as
\[
\frac{1}{n^2}
\sum_{(t_1,t_2)\in G}
\chi_{i'_1-i_1,\;i'_2-i_2}
\bigl((\mu+\sigma)(t_1,t_2)\bigr).
\]
Since \(\mu+\sigma\) is a permutation of \(G\), the character-sum property implies that this inner product is \(1\) when
\((i_1,i_2)=(i'_1,i'_2)\), and \(0\) otherwise. Hence the vectors corresponding to the cells of \(S_{(d_1,d_2)}\) form an orthonormal basis of \(\mathcal H_{n^2}\). Therefore, each \(S_{(d_1,d_2)}\) is a transversal, and
\(\{S_{(d_1,d_2)}:(d_1,d_2)\in G\}\) is a set of \(n^2\) mutually disjoint transversals of \(L\).

Finally, a cell belongs to \(T_{(c_1,c_2)}\cap S_{(d_1,d_2)}\) if and only if
\(
(\sigma-\operatorname{id})(i_1,i_2)
=
(c_1-d_1,c_2-d_2).
\)
Since \(\sigma-\operatorname{id}\) is a permutation of \(G\), this equation has a unique solution \((i_1,i_2)\in G\). Hence
$
|T_{(c_1,c_2)}\cap S_{(d_1,d_2)}|=1
$
for all \((c_1,c_2),(d_1,d_2)\in G\). Therefore, \(L\) is a maximal-cardinality \(2\text{-}\text{RQLS}(n^2)\).


\section{Main result}
In this section, we introduce a singular direct product construction for generating the \(1\text{-}\text{RQLS}(mn+h)\) with maximal cardinality. This approach extends the classical singular direct product method, which has proven  effective for classical Latin squares~\cite{Heinrich,Stinson}. The following definitions and notational conventions will be used throughout this construction.

\begin{definition}
An incomplete quantum Latin square, denoted by \(\text{IQLS}(m+h,h)\), is an \((m+h)\times(m+h)\) square whose \(h\times h\) subblock in the lower right
corner is empty, satisfying the following conditions:
\begin{itemize}
\item[(i)] Each nonempty entry is a unit column vector in the \((m+h)\)-dimensional Hilbert space \(\mathcal{H}_{m+h}\).

 \item[(ii)] For each \(i\in\{0,1,\ldots,m-1\}\), the entries in the \(i\)-th row (column) form an orthonormal basis of \(\mathcal{H}_{m+h}\).

\item[(iii)] For each \(j\in\{m,m+1,\ldots,m+h-1\}\), the nonempty entries in the \(j\)-th row (column) form an orthonormal basis of the subspace spanned by
the computational basis vectors \(|0\rangle,|1\rangle,\ldots,|m-1\rangle\in\mathcal{H}_{m+h}\).
\end{itemize}
\end{definition}

The cardinality \(c\) of an \(\text{IQLS}(m+h,h)\) is defined as the number of distinct vectors in the array. Clearly, $m+h\leq c\leq (m+h)^2-h^2$. An \(\text{IQLS}(m+h,h)\) has maximal cardinality if  $c=(m+h)^2-h^2$.

\begin{definition}
An incomplete transversal of an \(\text{IQLS}(m+h,h)\) is a set of \(m\) nonempty entries, one from each of the first \(m\) rows and columns, forming an orthonormal basis of the subspace spanned by the computational basis vectors \(|0\rangle,|1\rangle,\ldots,|m-1\rangle\in\mathcal{H}_{m+h}\).
\end{definition}

\begin{definition}
An \(\text{IQLS}(m+h,h)\) is called an incomplete resolvable quantum Latin square, denoted by \(\text{IRQLS}(m+h,h)\), if its nonempty entries can be
partitioned into \(m\) transversals and \(h\) incomplete transversals.
\end{definition}

Consider the following quantum state vectors:
$$
|a\rangle = \begin{pmatrix} a_0 \\ a_1 \\ \vdots \\ a_{n-1} \end{pmatrix} \in \mathcal{H}_n, \quad
|b\rangle = \begin{pmatrix} b_0 \\ b_1 \\ \vdots \\ b_{m-1} \end{pmatrix} \in \mathcal{H}_m, \quad
|c\rangle = \begin{pmatrix} c_0 \\ c_1 \\ \vdots \\ c_{m} \end{pmatrix} \in \mathcal{H}_{m+1}.
$$

\begin{definition}[Extended Tensor Product $\otimes_+$]
Let $h$ be a positive integer , we define the extended tensor product operation $\otimes_+$ as:
$$
|a\rangle \otimes_+ |b\rangle= \begin{pmatrix} |a\rangle \otimes |b\rangle \\ \mathbf{0}_h  \end{pmatrix} \in \mathcal{H}_{mn+h}.
$$
where $\mathbf{0}_h \in \mathbb{C}^h$ is the zero vector.
\end{definition}

\begin{definition}[Parameterized Tensor Product $\otimes_r$]
For each index $r \in \{0,1,\ldots,h-1\}$, we define the operation $\otimes_r$ as:
$$
|a\rangle \otimes_r |c\rangle= \begin{pmatrix} |a\rangle \otimes |c_{[m-1]}\rangle \\ c_m|r\rangle \end{pmatrix} \in \mathcal{H}_{mn+h},
$$
where $|c_{[m-1]}\rangle = (c_0, c_1, \ldots, c_{m-1})^\top$ is the projection onto the first $m$ components of $|c\rangle$, $|r\rangle \in \mathbb{C}^h$ is the $r$-th standard basis vector ($|r\rangle = (0,\ldots,0,1,0,\ldots,0)^\top$  with  $r$-th component equals to  $1$).
\end{definition}

\begin{construction}[Singular direct product construction]\label{con:Singular}  Let \( n, m, h \) be positive integers with \( n \geq h \). Suppose the following conditions are satisfied:
\begin{itemize}
\item[(i)] There exists a classical \(2\text{-}\text{RQLS}(n)\);
  \item[(ii)] There exist \( n - h \)   \(1\text{-}\text{RQLS}(m)\text{s}\) with maximal cardinality, such that all elements among them are mutually distinct;
  \item[(iii)] There exist \( h \) \(\text{IRQLS}(m+1,1)\text{s}\)  with maximal cardinality, such that all elements among them are mutually distinct;
  \item[(iv)] The vectors in these \( n - h \) \(1\text{-}\text{RQLS}(m)\text{s}\) are distinct from the \( m \)-dimensional projections of those  elements in the \(\text{IRQLS}(m+1, 1)\text{s}\) whose last  component is zero;
  \item[(v)] There exists a  \(1\text{-}\text{RQLS}(h)\) with maximal cardinality.
\end{itemize}
Then there exists a \(1\text{-}\text{RQLS}(mn+h)\) with maximal cardinality.
\end{construction}

\begin{proof}
Let \(A=(|a_{i,j}\rangle)\) be a classical \(2\text{-}\text{RQLS}(n)\) with two sets of mutually disjoint transversals \(\{T_p:p\in[n]\}\) and \(\{T'_q:q\in[n]\}\), satisfying \(|T_p\cap T'_q|=1\) for all \(p,q\in[n]\). For \((i,j)\in T_p\cap T'_q\), denote the corresponding entry \(|a_{i,j}\rangle\) by \(|a_{i,j,p,q}\rangle\).  Note that  \( |a_{i,j,p,q} \rangle \in \left\{|0 \rangle ,|1 \rangle ,\dots,|n-1 \rangle \right\} \).

Let $B^{(g)}=\bigl(|b^{(g)}_{k,l}\rangle\bigr)$, $g\in[n-h]$, be \(n-h\) maximal-cardinality \(1\text{-}\text{RQLS}(m)\text{s}\) such that all
vectors among them are mutually distinct. If \(|b^{(g)}_{k,l}\rangle\) is in the \(e\)-th transversal of \(B^{(g)}\), where \(e\in[m]\), denote it by \(|b^{(g)}_{k,l,e}\rangle\).

Let \(C^{(r)}=\bigl(|c^{(r)}_{s,t}\rangle\bigr)\), \(r\in[h]\), be \(h\) maximal-cardinality \(\text{IRQLS}(m+1,1)\text{s}\) such that all vectors among them are mutually distinct. If \(|c^{(r)}_{s,t}\rangle\) is in the \(w\)-th transversal of \(C^{(r)}\), where \(w\in[m]\), denote it by \(|c^{(r)}_{s,t,w}\rangle\). If \(|c^{(r)}_{s,t}\rangle\) belongs to the unique incomplete transversal of \(C^{(r)}\), denote it by \(|c^{(r)}_{s,t,\mathrm{I}}\rangle\).
For the purpose of our construction, we partition \( C^{(r)} \) into three distinct components: the upper-left \( m \times m\) submatrix  \( C^{(r)}_1 \), the lower-left \( 1 \times m \) submatrix   \( C^{(r)}_2 \), and the upper-right \( m \times 1 \) submatrix   \( C^{(r)}_3 \).

Let \(D=(|d_{u,v}\rangle)\) be a \(1\text{-}\text{RQLS}(h)\) with maximal cardinality. If \(|d_{u,v}\rangle\) belongs to the \(z\)-th transversal of \(D\), where \(z\in[h]\), denote it by \(|d_{u,v,z}\rangle\).

We construct a   matrix of order $mn+h$, denoted as $M$.  First, partition the subarray of order \(mn\) in the upper-left corner of \(M\) into \(n^2\) blocks of size \(m\times m\). For each \((i,j)\in[n]\times[n]\), define the \((i,j)\)-th block \(M_{i,j}\) according to which transversal in \(\{T_p:p\in[n]\}\) contains \((i,j)\), as follows:
\[
M_{i,j}=
\begin{cases}
|a_{i,j,g,q}\rangle\otimes_{+} B^{(g)}, & \text{if } (i,j)\in T_g,\quad g\in[n-h],\\[2mm]
|a_{i,j,n-h+r,q}\rangle\otimes_{r} C^{(r)}_1, & \text{if } (i,j)\in T_{n-h+r},\quad r\in[h],
\end{cases}
\]
where \(q\) is determined by \((i,j)\in T'_q\).
Next, consider the $h \times mn$ submatrix in the lower-left corner of $M$, which is partitioned into $hn$ blocks, each of size $1 \times m$. The $j$-th block in the \( (mn+r) \)-th row  is defined as
 $$M_{mn+r,j} = |a_{i,j,n-h+r,q}\rangle \otimes_{r} C^{(r)}_2.
$$
Similarly, the $mn \times h$ submatrix in the upper-right corner of $M$ is partitioned into $hn$ blocks, each of size $m \times 1$. The $i$-th block in the \( (mn+r) \)-th column is defined as
$$
M_{i,mn+r} = |a_{i,j,n-h+r,q}\rangle \otimes_{r} C^{(r)}_3.
$$
Finally, the  submatrix  of order $h$ in the lower-right corner of $M$ is denoted as $M_h$. For \(u,v\in[h]\), the entry in the \(u\)-th row and
\(v\)-th column of \(M_h\) is defined by $\begin{pmatrix} \mathbf{0}_{mn} \\ |d_{u,v}\rangle \end{pmatrix}$, which corresponds to the $(mn + u)$-th row and $(mn + v)$-th column of $M$. Note that the matrix \( M \) consists of elements from  \(|a_{i,j,g,q}\rangle\otimes_{+} B^{(g)}\), \(|a_{i,j,n-h+r,q}\rangle\otimes_{r} C^{(r)}\), and \(M_h\).

By the singular direct product construction in \cite{Zhang3}, the array \(M\) constructed above is a \(\text{QLS}(mn+h)\) with maximal cardinality \((mn+h)^2\). It remains to prove that \(M\) is resolvable. For this purpose, we define \(mn+h\) subsets of \(M\).

For each \(q\in[n]\) and \(e\in[m]\), define
\[
\begin{aligned}
\mathcal R_{q,e}
={}&
\bigcup_{g\in[n-h]}
\Bigl\{
|a_{i,j,g,q}\rangle\otimes_{+}|b^{(g)}_{k,l,e}\rangle :
(i,j)\in T_g\cap T'_q
\Bigr\}\\
&\cup
\bigcup_{r\in[h]}
\Bigl\{
|a_{i,j,n-h+r,q}\rangle\otimes_r|c^{(r)}_{s,t,e}\rangle :
(i,j)\in T_{n-h+r}\cap T'_q
\Bigr\}.
\end{aligned}
\]
where \(|b^{(g)}_{k,l,e}\rangle\) and \(|c^{(r)}_{s,t,e}\rangle\) range over all entries in the \(e\)-th transversals of \(B^{(g)}\) and \(C^{(r)}\), respectively. Since \(|T_p\cap T'_q|=1\) for all \(p,q\in[n]\), the transversal \(T'_q\) contains exactly one position from each \(T_p\). Hence, among
the \(n\) positions of \(T'_q\), \(n-h\) correspond to the blocks \(|a_{i,j,g,q}\rangle\otimes_{+}B^{(g)}\), each contributing \(m\) entries to \(\mathcal R_{q,e}\), while the remaining \(h\) positions correspond to the blocks \(|a_{i,j,n-h+r,q}\rangle\otimes_r C^{(r)}\), each contributing \(m+1\) entries. Since these entries lie in distinct blocks of \(M\), they are mutually disjoint. Therefore,
\[
|\mathcal R_{q,e}|=(n-h)m+h(m+1)=mn+h.
\]

Moreover, since \(T'_q\) is a transversal of \(A\), its \(n\) positions lie in distinct rows and columns. If \((i,j)\in T_g\), \(g\in[n-h]\), the \(e\)-th transversal of \(B^{(g)}\) yields \(m\) entries of \(M_{i,j}\) lying in distinct rows and columns. If \((i,j)\in T_{n-h+r}\), \(r\in[h]\), the \(e\)-th transversal of \(C^{(r)}\) yields \(m+1\) entries of \(M\), of which \(m\) lie in the rows and columns associated with \(M_{i,j}\), while the remaining one lies in the \((mn+r)\)-th row and the \((mn+r)\)-th column. Hence \(\mathcal R_{q,e}\) contains exactly one entry in each row and each column of \(M\).

For each \(r\in[h]\), define
\[
\begin{aligned}
\mathcal S_r
={}&
\bigcup_{q\in[n]}
\Bigl\{
|a_{i,j,n-h+r,q}\rangle\otimes_r
|c^{(r)}_{s,t,\mathrm{I}}\rangle :
(i,j)\in T_{n-h+r}\cap T'_q
\Bigr\} \cup
\Biggl\{
\begin{pmatrix}
\mathbf{0}_{mn}\\
|d_{u,v,r}\rangle
\end{pmatrix}
\Biggr\},
\end{aligned}
\]
where \(|c^{(r)}_{s,t,\mathrm{I}}\rangle\) ranges over all entries in the incomplete transversal of \(C^{(r)}\), and \(|d_{u,v,r}\rangle\) ranges over all entries in the \(r\)-th transversal of \(D\). Since the transversals \(T'_q\), \(q\in[n]\), are mutually disjoint, the intersections \(T_{n-h+r}\cap T'_q\) determine \(n\) distinct blocks of \(M\). Each block contributes \(m\) entries determined by the incomplete transversal of \(C^{(r)}\), yielding a total of \(nm\) entries. These entries lie in the subarray of order \(mn\) in the upper-left corner of \(M\), whereas the \(r\)-th transversal of \(D\) contributes \(h\) entries in \(M_h\). Hence
$
|\mathcal S_r|=nm+h.
$

Since \(T_{n-h+r}\) is a transversal of \(A\), its \(n\) positions lie in distinct rows and columns. Moreover, the incomplete transversal of \(C^{(r)}\) contains exactly one entry in each of its first \(m\) rows and columns. Thus the first part of \(\mathcal S_r\) contains exactly one entry in each of the first \(mn\) rows and columns of \(M\), while the \(r\)-th transversal of \(D\) contributes exactly one entry in each of the remaining \(h\) rows and columns. Therefore, \(\mathcal S_r\) contains exactly one entry in each row and each column of \(M\).

Thus, we obtain \(nm+h\) subsets
$
\{\mathcal R_{q,e}:q\in[n],\,e\in[m]\}
\cup
\{\mathcal S_r:r\in[h]\}.
$
We next show that these \(mn+h\) subsets are mutually disjoint and partition the entries of \(M\). For a fixed \(q\), the sets \(\mathcal R_{q,e}\), \(e\in[m]\), are mutually disjoint since the transversals of each \(B^{(g)}\) and \(C^{(r)}\) are mutually disjoint. For distinct \(q,q'\in[n]\), the transversals \(T'_q\) and \(T'_{q'}\) are disjoint, and hence \(\mathcal R_{q,e}\) and \(\mathcal R_{q',e'}\) are disjoint.  Moreover, the incomplete transversal of each \(C^{(r)}\) is disjoint from its \(m\) transversals, so every \(\mathcal S_r\) is disjoint from all \(\mathcal R_{q,e}\). The sets \(\mathcal S_r\), \(r\in[h]\), are also mutually disjoint.

Furthermore, the \(m\) transversals of each \(B^{(g)}\) partition its entries, the \(m\) transversals together with the incomplete transversal of each \(C^{(r)}\) partition its nonempty entries, and the \(h\) transversals of \(D\) partition its entries. Hence
$
\{\mathcal R_{q,e}:q\in[n],\,e\in[m]\}
\cup
\{\mathcal S_r:r\in[h]\}
$
forms a partition of the entries of \(M\).

It remains to verify that each \(\mathcal R_{q,e}\), \(q\in[n]\), \(e\in[m]\), and each \(\mathcal S_r\), \(r\in[h]\), is a transversal of \(M\).   This is verified in Appendix~A. Therefore, these \(mn+h\) mutually disjoint transversals partition the entries of \(M\), and hence \(M\) is resolvable. Thus, \(M\) is a \(1\text{-}\text{RQLS}(mn+h)\) with maximal cardinality.
\end{proof}

\begin{lemma}{\rm (\cite{Zhang1})}\label{lem:abca}
Let \(|a\rangle,|b\rangle\in\mathcal{H}_{m}\) , and \(|c\rangle,|d\rangle\in\mathcal{H}_{n}\) be unit vectors. Then \(|a\rangle \otimes |c\rangle\) and \(|b\rangle \otimes |d\rangle\) are identical if and only if \(|a\rangle\) is identical to \(|b\rangle\) and \(|c\rangle \) is identical to \(|d\rangle\). \end{lemma}

\begin{lemma}{\rm (\cite{Zhang6})}\label{lem:disjoint-QLSs}
Suppose that \(m\geq3\). For any positive integer \(k\), the following statements hold:

{\rm (1)} If there exists a \(\text{QLS}(m)\) with cardinality \(c\), then there exist \(k\) \(\text{QLS}(m)\text{s}\), each of cardinality \(c\),
such that all elements among them are mutually distinct.

{\rm (2)} If there exists a \(\text{QLS}(m+1)\) with cardinality \(c\), then there exist \(k\) \(\text{QLS}(m+1)\text{s}\), each of cardinality \(c\),
such that the intersection of any two of them is exactly \(\{|m\rangle\}\).
\end{lemma}

\begin{lemma}\label{lem:conditions-234}
Suppose that \(m\geq3\), and let \(n\) and \(h\) be positive integers with \(n\geq h\). If there exist a maximal-cardinality \(1\text{-}\text{RQLS}(m)\) and a maximal-cardinality  \(1\text{-}\text{RQLS}(m+1)\), then conditions {\rm (ii)--(iv)} of Construction~\ref{con:Singular} can be satisfied.
\end{lemma}

\begin{proof}
In the proof of Lemma~\ref{lem:disjoint-QLSs}, the quantum Latin squares  are obtained from a given
quantum Latin square by applying suitable unitary transformations to all its entries.  If the given square is resolvable, then the resulting squares are also resolvable, since the vectors in each transversal are transformed by the same unitary matrix and hence still form an orthonormal basis.  It follows from Lemma~\ref{lem:disjoint-QLSs}{\rm (1)} that, if there exists a maximal-cardinality \(1\text{-}\text{RQLS}(m)\), then for any positive integer \(k\), there exist \(k\) maximal-cardinality \(1\text{-}\text{RQLS}(m)\text{s}\) such that all elements among them are mutually distinct.

By Lemma~\ref{lem:disjoint-QLSs}{\rm (2)}, there exist \(h\) maximal-cardinality \(1\text{-}\text{RQLS}(m+1)\text{s}\) such that the intersection of any two of them is exactly \(\{|m\rangle\}\). For each square, permute its rows and columns, if necessary, so that \(|m\rangle\) lies in the lower-right corner, and then delete this entry. The resulting array is an \(\text{IRQLS}(m+1,1)\), since the transversal containing \(|m\rangle\) becomes an incomplete transversal after
\(|m\rangle\) is deleted, while the other \(m\) transversals remain unchanged. Moreover, its cardinality is \((m+1)^2-1\), and hence is maximal. Since the common element \(|m\rangle\) has been deleted, none of the resulting \(\text{IRQLS}(m+1,1)\text{s}\) contains \(|m\rangle\).
Therefore, we obtain \(h\) maximal-cardinality \(\text{IRQLS}(m+1,1)\text{s}\) such that all elements among them are mutually distinct.

There are only finitely many \(m\)-dimensional projections of those elements in the \(h\) maximal-cardinality \(\text{IRQLS}(m+1,1)\text{s}\) whose last component is zero. By Lemma~\ref{lem:disjoint-QLSs}{\rm (1)}, there exist arbitrarily many maximal-cardinality \(1\text{-}\text{RQLS}(m)\text{s}\) whose elements are
mutually distinct. Since these \(1\text{-}\text{RQLS}(m)\text{s}\) are mutually disjoint, each of the above projections can occur in at most one of them. Therefore, only finitely many of these \(1\text{-}\text{RQLS}(m)\text{s}\) contain one of the above projections. Hence we can choose \(n-h\) maximal-cardinality
\(1\text{-}\text{RQLS}(m)\text{s}\) whose elements are mutually distinct and are also distinct from all these projections.

Therefore, conditions {\rm (ii)--(iv)} of Construction~\ref{con:Singular} can be satisfied simultaneously.
\end{proof}

\begin{lemma}\label{lem:RQLS-90}
There exists a \(1\text{-}\text{RQLS}(90)\) with maximal cardinality.
\end{lemma}

\begin{proof}
Let \(A=(|a_{i,j}\rangle)\) be a classical \(1\text{-}\text{RQLS}(10)\), and let \(\{T_g:g\in[10]\}\) be a partition of its entries into ten transversals. If \((i,j)\in T_g\), denote \(|a_{i,j}\rangle\) by \(|a_{i,j,g}\rangle\). By Theorem~\ref{th:1.1}{\rm (2)}, there exists a maximal-cardinality \(1\text{-}\text{RQLS}(9)\). By Lemma~\ref{lem:disjoint-QLSs}{\rm (1)}, together with the fact that unitary transformations preserve resolvability, there exist ten maximal-cardinality \(1\text{-}\text{RQLS}(9)\text{s}\) $B^{(g)}=(|b^{(g)}_{k,l}\rangle)$,  $g\in[10]$, such that all elements among them are mutually distinct.

We construct a \(90\times90\) array \(M\), partitioned into
\(10^2\) blocks of size \(9\times9\). For \(i,j\in[10]\), suppose
that \((i,j)\in T_g\). Define the \((i,j)\)-th block of \(M\) by
\[
M_{i,j}
=
|a_{i,j,g}\rangle\otimes B^{(g)}
=
\bigl(
|a_{i,j,g}\rangle\otimes|b^{(g)}_{k,l}\rangle
\bigr)_{k,l\in[9]}.
\]

By the standard direct-product argument, every row and column of \(M\) forms an orthonormal basis of \(\mathcal H_{10}\otimes\mathcal H_9=\mathcal H_{90}\).
Hence \(M\) is a \(\text{QLS}(90)\).  Suppose that two entries
$|a_{i,j,g}\rangle\otimes|b^{(g)}_{k,l}\rangle \quad\text{and}\quad |a_{i',j',g'}\rangle\otimes|b^{(g')}_{k',l'}\rangle $
are identical. By Lemma~\ref{lem:abca}, their corresponding factors
are identical. Since the elements among the \(B^{(g)}\text{s}\) are
mutually distinct, we have \(g=g'\). As \(T_g\) is a transversal of
\(A\), this implies \((i,j)=(i',j')\). Since \(B^{(g)}\) has maximal
cardinality, we further obtain \((k,l)=(k',l')\). Thus the two entries
occupy the same position in \(M\). Hence all \(90^2\) entries of \(M\)
are mutually distinct, and \(M\) has maximal cardinality.

It remains to show that \(M\) is resolvable. For \(g\in[10]\) and \(e\in[9]\), let \(S^{(g)}_e\) denote the \(e\)-th transversal of \(B^{(g)}\), and define \[ \mathcal T_{g,e} = \bigl\{ |a_{i,j,g}\rangle\otimes|b^{(g)}_{k,l}\rangle: (i,j)\in T_g,\ (k,l)\in S^{(g)}_e \bigr\}. \] Each \(\mathcal T_{g,e}\) consists of \(90\) entries lying in distinct rows and columns of \(M\). If two entries of \(\mathcal T_{g,e}\) arise from distinct positions of \(T_g\), then their first factors are orthogonal. If they arise from the same position of \(T_g\), then their second factors are distinct elements of the transversal \(S^{(g)}_e\), and hence are orthogonal. Therefore, the entries of \(\mathcal T_{g,e}\) form an orthonormal basis of \(\mathcal H_{90}\), and thus \(\mathcal T_{g,e}\) is a transversal of \(M\). Moreover, since the transversals \(T_g\), \(g\in[10]\), partition the entries of \(A\), and the transversals \(S^{(g)}_e\), \(e\in[9]\), partition the entries of \(B^{(g)}\), the \(90\) transversals $\{\mathcal T_{g,e}:g\in[10],\,e\in[9]\}$ are mutually disjoint and partition the entries of \(M\).  Therefore, \(M\) is a \(1\text{-}\text{RQLS}(90)\) with maximal cardinality.
\end{proof}

We are now ready to prove our main result.

\textbf{Proof of Theorem~\ref{th:1.2}.} By Theorem~\ref{th:1.1}{\rm (2)}, there exists a maximal-cardinality \(1\text{-}\text{RQLS}(v)\) for every odd integer \(v\geq7\). By Lemma~\ref{th:max-RQLS-4k}, there also exists a maximal-cardinality \(1\text{-}\text{RQLS}(v)\) for every \(v\geq8\) with
\(v\equiv0\pmod4\). Hence it remains to consider \(v\equiv2\pmod4\).

We apply Construction~\ref{con:Singular} with \(m=7\). Since maximal-cardinality \(1\text{-}\text{RQLS}(7)\) and \(1\text{-}\text{RQLS}(8)\) exist,
Lemma~\ref{lem:conditions-234} implies that conditions
{\rm (ii)--(iv)} of Construction~\ref{con:Singular} can be satisfied.
For each residue class of \(v\) modulo \(7\), choose \(h\) according
to the following table:
\[
\begin{array}{c|ccccccc}
v\pmod7 & 0 & 1 & 2 & 3 & 4 & 5 & 6\\
\hline
h       & 7 & 1 & 9 & 17 & 11 & 12 & 13
\end{array}
\]
and set \(n=(v-h)/7\), so that \(v=7n+h\). By
Theorem~\ref{th:1.1}{\rm (2)}, maximal-cardinality
\(1\text{-}\text{RQLS}(h)\text{s}\) exist for
\(h\in\{7,9,11,13,17\}\), while the case \(h=12\) follows from
Lemma~\ref{th:max-RQLS-4k}, and the case \(h=1\) is trivial.
Hence condition~{\rm (v)} of Construction~\ref{con:Singular} is
satisfied.

It remains to verify condition~{\rm (i)}. For the above choices of
\(h\), this condition is satisfied whenever \(n\geq h\) and
\(n\notin\{2,3,6,10\}\). Restricting to
\(v\equiv2\pmod4\), the orders obtained from
Construction~\ref{con:Singular} and the remaining orders are
summarized as follows:
\[
\begin{array}{c|c|c|c}
h
& v=7n+h
& \text{orders obtained}
& \text{remaining orders}\\
\hline
1  & v=7n+1  & 50,\quad v\geq78  & 22\\
7  & v=7n+7  & v\geq70           & 14,42\\
9  & v=7n+9  & v\geq86           & 30,58\\
11 & v=7n+11 & v\geq102          & 18,46,74\\
12 & v=7n+12 & v\geq110          & 26,54,82\\
13 & v=7n+13 & v\geq118          & 34,62,90\\
17 & v=7n+17 & v\geq150          & 10,38,66,94,122
\end{array}
\]

Among the remaining orders listed in the table, the case \(v=90\) has already been settled by Lemma~\ref{lem:RQLS-90}. The order \(122\) can be further settled by applying Construction~\ref{con:Singular} with \(m=n=11\) and \(h=1\). By Theorem~\ref{th:1.1}{\rm (2)} and Lemma~\ref{th:max-RQLS-4k}, maximal-cardinality \(1\text{-}\text{RQLS}(11)\) and \(1\text{-}\text{RQLS}(12)\) exist, respectively. Hence Lemma~\ref{lem:conditions-234} gives conditions {\rm (ii)--(iv)}, while conditions {\rm (i)} and {\rm (v)} follow from \(11\notin\{2,3,6,10\}\) and \(h=1\), respectively.

Therefore, for every integer \(v\geq7\) with
\(v\notin\{10,14,18,22,26,30,34,38,42,46,54,58,62,66,\)\\\(74,82,94\}\),
there exists a maximal-cardinality \(1\text{-}\text{RQLS}(v)\). This completes the proof.

\section{Concluding remarks}

In this paper, we have introduced \(t\)-resolvable quantum Latin squares and established their equivalence with a class of mutually orthogonal quantum Latin squares. We have also investigated the existence of maximal-cardinality \(1\text{-}\text{RQLS}(v)\text{s}\) and obtained several existence results for maximal-cardinality \(2\text{-}\text{RQLS}(v)\text{s}\).

The results presented in this paper can also address some of the open problems left in the work by Huang and Li \cite{Liyang}. A \(\text{QLS}(v)\) is called {\it idempotent} if its main diagonal is a transversal. They proved that, for every integer \(v\geq6\), there exists a maximal-cardinality idempotent \(\text{QLS}(v)\), with possible exceptional orders \(M_1=\{6,8,10,12,14,18,20,\)\(24,26,30,32,38,62\}\).
They further defined a {\it diagonal quantum Latin square}, denoted by \(\text{DQLS}(v)\), to be a \(\text{QLS}(v)\) whose main diagonal and
anti-diagonal are both transversals. For even \(v\geq6\), they showed that a maximal-cardinality \(\text{DQLS}(v)\) exists, with possible exceptional
orders   \(M_2=\{6,8,10,12,14,18,20,22,24,26,30,32,34,38,40,46, 48,50,58,62,74,82,94,122\}\).

Combining Lemma~\ref{th:max-RQLS-4k} with the results of \cite{Liyang}, the possible exceptional orders for maximal-cardinality idempotent quantum Latin squares are reduced to \(M_1\setminus\{8,12,20,24,32\}\), while the possible even exceptional orders for maximal-cardinality diagonal quantum Latin squares are reduced to \(M_2\setminus\{8,12,20,24,32,40,48\}\).

For maximal-cardinality \(1\text{-}\text{RQLS}(v)\text{s}\), our main existence result leaves only \(17\) possible exceptional orders. It remains to determine whether a maximal-cardinality \(1\text{-}\text{RQLS}(v)\) exists for
$v\in\{10,14,18,22,26,30,34,38,42,46,54,58,62,66,74,82,94\}$.

For maximal-cardinality \(2\text{-}\text{RQLS}(v)\text{s}\), the known result at present are  \(\gcd(v,6)=1\) and \(v\notin\{5,7,11\}\) or \(v=n^2\) for all \(n\geq6\). It remains to determine the existence of maximal-cardinality \(2\text{-}\text{RQLS}(v)\text{s}\) for the remaining orders.

Finally, by Corollary~\ref{cor:MOQLS}, every maximal-cardinality \(1\text{-}\text{RQLS}(v)\) obtained in this paper gives rise to a non-classical \(2\text{-}\text{MOQLS}(v)\) consisting of one maximal-cardinality quantum Latin square and one classical quantum Latin square. Likewise, every maximal-cardinality \(2\text{-}\text{RQLS}(v)\) yields a non-classical \(3\text{-}\text{MOQLS}(v)\) consisting of one maximal-cardinality
quantum Latin square and two classical quantum Latin squares. For the case $t\ge 3$, we have not yet found an example of a maximal-cardinality \(t\text{-}\text{RQLS}(v)\), which also constitutes an interesting open problem.
~

\noindent{\bf Acknowledgments} The authors would like to express their sincere gratitude to  Prof. Lie Zhu (Soochow University) for his expert guidance and constructive discussions. H. Cao's research was supported by the National Natural Science Foundation of China (Grants No. 12471313 and No. 12071226). Y.Zhang's research was supported by the Postgraduate Research \& Practice Innovation Program of Jiangsu Province (No.~26CXJH3078).

~

\noindent{\bf Appendix A. \(M\) is Resolvable in Construction~\ref{con:Singular}}

\begin{proof}
We first show that each \(\mathcal R_{q,e}\), \(q\in[n]\) and \(e\in[m]\), is a transversal of \(M\). As shown in the proof of Construction~\ref{con:Singular}, \(\mathcal R_{q,e}\) contains \(mn+h\) entries lying in distinct rows and columns of \(M\). It remains to show that its entries are mutually orthogonal.
Let \(\alpha\) and \(\beta\) be two distinct entries of \(\mathcal R_{q,e}\). We distinguish the following five cases.

\textbf{Case 1}. \(\alpha=|a_{i,j,g,q}\rangle\otimes_{+}|b^{(g)}_{k,l,e}\rangle\) and \(\beta=|a_{i,j,g,q}\rangle\otimes_{+}|b^{(g)}_{k',l',e}\rangle\),
where \((k,l)\neq(k',l')\).

 Since \(|b^{(g)}_{k,l,e}\rangle\) and \(|b^{(g)}_{k',l',e}\rangle\) belong to the \(e\)-th transversal of \(B^{(g)}\), \(\bigl(|b^{(g)}_{k,l,e}\rangle, |b^{(g)}_{k',l',e}\rangle\bigr)=0\). Hence
\[
\begin{aligned}
(\alpha,\beta)
&=
\left(
\begin{pmatrix}
|a_{i,j,g,q}\rangle\otimes|b^{(g)}_{k,l,e}\rangle\\
\mathbf{0}_h
\end{pmatrix},
\begin{pmatrix}
|a_{i,j,g,q}\rangle\otimes|b^{(g)}_{k',l',e}\rangle\\
\mathbf{0}_h
\end{pmatrix}
\right)\\
&=
\bigl(|a_{i,j,g,q}\rangle,|a_{i,j,g,q}\rangle\bigr)
\bigl(|b^{(g)}_{k,l,e}\rangle,
|b^{(g)}_{k',l',e}\rangle\bigr)\\
&=0.
\end{aligned}
\]

\textbf{Case 2}. \(\alpha=|a_{i,j,n-h+r,q}\rangle\otimes_r|c^{(r)}_{s,t,e}\rangle\) and \(\beta=|a_{i,j,n-h+r,q}\rangle\otimes_r|c^{(r)}_{s',t',e}\rangle\),
where \((s,t)\neq(s',t')\).

The notation \(|c^{(r)}_{s,t,e,[m-1]}\rangle\) denotes the projection of \(|c^{(r)}_{s,t,e}\rangle\) onto its first \(m\) components, while \(c^{(r)}_{s,t,e,m}\) denotes its last component. Since \(|c^{(r)}_{s,t,e}\rangle\) and \(|c^{(r)}_{s',t',e}\rangle\) belong to the \(e\)-th transversal of \(C^{(r)}\), \(\bigl(|c^{(r)}_{s,t,e}\rangle, |c^{(r)}_{s',t',e}\rangle\bigr)=0\). Thus
\[
\begin{aligned}
(\alpha,\beta)
&=
\left(
\begin{pmatrix}
|a_{i,j,n-h+r,q}\rangle\otimes
|c^{(r)}_{s,t,e,[m-1]}\rangle\\
c^{(r)}_{s,t,e,m}|r\rangle
\end{pmatrix},
\begin{pmatrix}
|a_{i,j,n-h+r,q}\rangle\otimes
|c^{(r)}_{s',t',e,[m-1]}\rangle\\
c^{(r)}_{s',t',e,m}|r\rangle
\end{pmatrix}
\right)\\
&=
\bigl(|a_{i,j,n-h+r,q}\rangle,
|a_{i,j,n-h+r,q}\rangle\bigr)
\bigl(|c^{(r)}_{s,t,e,[m-1]}\rangle,
|c^{(r)}_{s',t',e,[m-1]}\rangle\bigr)
+\overline{c^{(r)}_{s,t,e,m}}\,c^{(r)}_{s',t',e,m}\\
&=
\bigl(|c^{(r)}_{s,t,e}\rangle,
|c^{(r)}_{s',t',e}\rangle\bigr)\\
&=0.
\end{aligned}
\]

\textbf{Case 3}.  \(\alpha=|a_{i,j,g,q}\rangle\otimes_{+}|b^{(g)}_{k,l,e}\rangle\) and \(\beta=|a_{i',j',g',q}\rangle\otimes_{+}|b^{(g')}_{k',l',e}\rangle\), where \((i,j)\neq(i',j')\).

Since \((i,j)\) and \((i',j')\) are two distinct positions in the transversal \(T'_q\) of \(A\), we have \(\bigl(|a_{i,j,g,q}\rangle, |a_{i',j',g',q}\rangle\bigr)=0\). Hence
\[
\begin{aligned}
(\alpha,\beta)
&=
\left(
\begin{pmatrix}
|a_{i,j,g,q}\rangle\otimes|b^{(g)}_{k,l,e}\rangle\\
\mathbf{0}_h
\end{pmatrix},
\begin{pmatrix}
|a_{i',j',g',q}\rangle\otimes|b^{(g')}_{k',l',e}\rangle\\
\mathbf{0}_h
\end{pmatrix}
\right)\\
&=
\bigl(|a_{i,j,g,q}\rangle,
|a_{i',j',g',q}\rangle\bigr)
\bigl(|b^{(g)}_{k,l,e}\rangle,
|b^{(g')}_{k',l',e}\rangle\bigr)\\
&=0.
\end{aligned}
\]

\textbf{Case 4}. \(\alpha=|a_{i,j,n-h+r,q}\rangle\otimes_r|c^{(r)}_{s,t,e}\rangle\) and \(\beta=|a_{i',j',n-h+r',q}\rangle\otimes_{r'}|c^{(r')}_{s',t',e}\rangle\), where \((i,j)\neq(i',j')\).

Since \((i,j)\) and \((i',j')\) are distinct positions in \(T'_q\), \(\bigl(|a_{i,j,n-h+r,q}\rangle, |a_{i',j',n-h+r',q}\rangle\bigr)=0\).
Moreover, since \((i,j)\in T_{n-h+r}\cap T'_q\) and \((i',j')\in T_{n-h+r'}\cap T'_q\), the condition \(|T_p\cap T'_q|=1\) implies \(r\neq r'\), and hence \((|r\rangle,|r'\rangle)=0\). Therefore,
\[
\begin{aligned}
(\alpha,\beta)
&=
\left(
\begin{pmatrix}
|a_{i,j,n-h+r,q}\rangle\otimes
|c^{(r)}_{s,t,e,[m-1]}\rangle\\
c^{(r)}_{s,t,e,m}|r\rangle
\end{pmatrix},
\begin{pmatrix}
|a_{i',j',n-h+r',q}\rangle\otimes
|c^{(r')}_{s',t',e,[m-1]}\rangle\\
c^{(r')}_{s',t',e,m}|r'\rangle
\end{pmatrix}
\right)\\
&=
\bigl(|a_{i,j,n-h+r,q}\rangle,
|a_{i',j',n-h+r',q}\rangle\bigr)
\bigl(|c^{(r)}_{s,t,e,[m-1]}\rangle,
|c^{(r')}_{s',t',e,[m-1]}\rangle\bigr)\\
&=0.
\end{aligned}
\]

\textbf{Case 5}. \(\alpha=|a_{i,j,g,q}\rangle\otimes_{+}|b^{(g)}_{k,l,e}\rangle\) and
\(\beta=|a_{i',j',n-h+r,q}\rangle\otimes_r|c^{(r)}_{s,t,e}\rangle\) , where \((i,j)\neq(i',j')\).

Since \((i,j)\) and \((i',j')\) are   distinct positions in  \(T'_q\), \(\bigl(|a_{i,j,g,q}\rangle, |a_{i',j',n-h+r,q}\rangle\bigr)=0\). Hence
\[
\begin{aligned}
(\alpha,\beta)
&=
\left(
\begin{pmatrix}
|a_{i,j,g,q}\rangle\otimes|b^{(g)}_{k,l,e}\rangle\\
\mathbf{0}_h
\end{pmatrix},
\begin{pmatrix}
|a_{i',j',n-h+r,q}\rangle\otimes
|c^{(r)}_{s,t,e,[m-1]}\rangle\\
c^{(r)}_{s,t,e,m}|r\rangle
\end{pmatrix}
\right)\\
&=
\bigl(|a_{i,j,g,q}\rangle,
|a_{i',j',n-h+r,q}\rangle\bigr)
\bigl(|b^{(g)}_{k,l,e}\rangle,
|c^{(r)}_{s,t,e,[m-1]}\rangle\bigr)\\
&=0.
\end{aligned}
\]

Hence the entries of \(\mathcal R_{q,e}\) are mutually orthogonal. Since \(|\mathcal R_{q,e}|=mn+h\), they form an orthonormal basis of \(\mathcal H_{mn+h}\). Therefore, \(\mathcal R_{q,e}\) is a transversal of \(M\).

We next show that each \(\mathcal S_r\), \(r\in[h]\), is a transversal of \(M\). As shown in the proof of Construction~\ref{con:Singular}, \(\mathcal S_r\) contains \(mn+h\) entries lying in distinct rows and columns of \(M\). It remains to show that its entries are mutually orthogonal. Let \(\alpha\) and \(\beta\) be two distinct entries of \(\mathcal S_r\). We distinguish the following four cases.

\textbf{Case 6}.
\(\alpha=|a_{i,j,n-h+r,q}\rangle\otimes_r
|c^{(r)}_{s,t,\mathrm I}\rangle\) and
\(\beta=|a_{i,j,n-h+r,q}\rangle\otimes_r
|c^{(r)}_{s',t',\mathrm I}\rangle\), where
\((s,t)\neq(s',t')\).

Since \(|c^{(r)}_{s,t,\mathrm I}\rangle\) and \(|c^{(r)}_{s',t',\mathrm I}\rangle\) belong to the incomplete transversal of \(C^{(r)}\), $\bigl(|c^{(r)}_{s,t,\mathrm I}\rangle,|c^{(r)}_{s',t',\mathrm I}\rangle\bigr)=0.$ Moreover, \(c^{(r)}_{s,t,\mathrm I,m}=c^{(r)}_{s',t',\mathrm I,m}=0\). Hence
\[
\begin{aligned}
(\alpha,\beta)
&=
\left(
\begin{pmatrix}
|a_{i,j,n-h+r,q}\rangle\otimes
|c^{(r)}_{s,t,\mathrm I,[m-1]}\rangle\\
\mathbf{0}_h
\end{pmatrix},
\begin{pmatrix}
|a_{i,j,n-h+r,q}\rangle\otimes
|c^{(r)}_{s',t',\mathrm I,[m-1]}\rangle\\
\mathbf{0}_h
\end{pmatrix}
\right)\\
&=
\bigl(|a_{i,j,n-h+r,q}\rangle,
|a_{i,j,n-h+r,q}\rangle\bigr)
\bigl(|c^{(r)}_{s,t,\mathrm I,[m-1]}\rangle,
|c^{(r)}_{s',t',\mathrm I,[m-1]}\rangle\bigr)\\
&=
\bigl(|c^{(r)}_{s,t,\mathrm I}\rangle,
|c^{(r)}_{s',t',\mathrm I}\rangle\bigr)\\
&=0.
\end{aligned}
\]

\textbf{Case 7}.
\(\alpha=|a_{i,j,n-h+r,q}\rangle\otimes_r
|c^{(r)}_{s,t,\mathrm I}\rangle\) and
\(\beta=|a_{i',j',n-h+r,q'}\rangle\otimes_r
|c^{(r)}_{s',t',\mathrm I}\rangle\), where
\((i,j)\neq(i',j')\).

Since \((i,j)\) and \((i',j')\) are distinct positions in \(T_{n-h+r}\), $\bigl(|a_{i,j,n-h+r,q}\rangle,|a_{i',j',n-h+r,q'}\rangle\bigr)=0$. Moreover, \(c^{(r)}_{s,t,\mathrm I,m} =c^{(r)}_{s',t',\mathrm I,m}=0\). Hence
\[
\begin{aligned}
(\alpha,\beta)
&=
\left(
\begin{pmatrix}
|a_{i,j,n-h+r,q}\rangle\otimes
|c^{(r)}_{s,t,\mathrm I,[m-1]}\rangle\\
\mathbf{0}_h
\end{pmatrix},
\begin{pmatrix}
|a_{i',j',n-h+r,q'}\rangle\otimes
|c^{(r)}_{s',t',\mathrm I,[m-1]}\rangle\\
\mathbf{0}_h
\end{pmatrix}
\right)\\
&=
\bigl(|a_{i,j,n-h+r,q}\rangle,
|a_{i',j',n-h+r,q'}\rangle\bigr)
\bigl(|c^{(r)}_{s,t,\mathrm I,[m-1]}\rangle,
|c^{(r)}_{s',t',\mathrm I,[m-1]}\rangle\bigr)\\
&=0.
\end{aligned}
\]

\textbf{Case 8}.
\(\alpha=
\begin{pmatrix}
\mathbf{0}_{mn}\\
|d_{u,v,r}\rangle
\end{pmatrix}\) and
\(\beta=
\begin{pmatrix}
\mathbf{0}_{mn}\\
|d_{u',v',r}\rangle
\end{pmatrix}\), where
\((u,v)\neq(u',v')\).

Since \(|d_{u,v,r}\rangle\) and \(|d_{u',v',r}\rangle\) belong to the
\(r\)-th transversal of \(D\),
\(\bigl(|d_{u,v,r}\rangle,|d_{u',v',r}\rangle\bigr)=0\). Hence
\[
\begin{aligned}
(\alpha,\beta)
&=
\left(
\begin{pmatrix}
\mathbf{0}_{mn}\\
|d_{u,v,r}\rangle
\end{pmatrix},
\begin{pmatrix}
\mathbf{0}_{mn}\\
|d_{u',v',r}\rangle
\end{pmatrix}
\right)\\
&=
\bigl(|d_{u,v,r}\rangle,
|d_{u',v',r}\rangle\bigr)\\
&=0.
\end{aligned}
\]

\textbf{Case 9}.
\(\alpha=|a_{i,j,n-h+r,q}\rangle\otimes_r
|c^{(r)}_{s,t,\mathrm I}\rangle\) and
\(\beta=
\begin{pmatrix}
\mathbf{0}_{mn}\\
|d_{u,v,r}\rangle
\end{pmatrix}\).

Since \(c^{(r)}_{s,t,\mathrm I,m}=0\), we have
\[
\begin{aligned}
(\alpha,\beta)
&=
\left(
\begin{pmatrix}
|a_{i,j,n-h+r,q}\rangle\otimes
|c^{(r)}_{s,t,\mathrm I,[m-1]}\rangle\\
\mathbf{0}_h
\end{pmatrix},
\begin{pmatrix}
\mathbf{0}_{mn}\\
|d_{u,v,r}\rangle
\end{pmatrix}
\right)\\
&=0.
\end{aligned}
\]

Hence the entries of \(\mathcal S_r\) are mutually orthogonal. Since \(|\mathcal S_r|=mn+h\), they form an orthonormal basis of
\(\mathcal H_{mn+h}\). Therefore, \(\mathcal S_r\) is a transversal of \(M\).

Thus, \(\mathcal R_{q,e}\) is a transversal of \(M\) for every \(q\in[n]\) and \(e\in[m]\), and \(\mathcal S_r\) is a transversal of \(M\) for every \(r\in[h]\). Since these \(mn+h\) transversals are mutually disjoint and partition the entries of \(M\), \(M\) is resolvable. Hence, \(M\) is a \(1\text{-}\text{RQLS}(mn+h)\) with maximal cardinality.
\end{proof}

\end{document}